\documentclass[10pt]{amsart}
\usepackage{amssymb,mathtools,amsthm}
\usepackage{cite}

\newtheorem{thm}{Theorem}[section]
\newtheorem{lemma}{Lemma}[section]

\newcommand{\R}{\mathbb R}
\newcommand{\eps}{\varepsilon}
\newcommand{\dd}{\, \mathrm{d}}
\newcommand{\abs}[1]{\left\lvert #1 \right\rvert}
\newcommand{\norm}[1]{\left\lVert #1 \right\rVert}
\newcommand{\ddt}{\frac{\mathrm{d}}{\mathrm{d}t}}

\numberwithin{equation}{section}

\title[$L^p$ regularity of solutions to the gHS system]{On the $L^p$ regularity of solutions to the generalized Hunter-Saxton system}

\author[Choi]{Jaeho Choi}
\address{Department of Mathematics and Statistics, Williams College, Williamstown, MA 01267, USA}
\email{jc14@williams.edu}
\thanks{}
\author[Krishna]{Nitin Krishna} 
\address{Department of Mathematics, The University of Chicago, Chicago, IL 60637, USA}
\email{nakrishna@uchicago.edu}
\thanks{}
\author[Magill]{Nicole Magill} 
\address{Department of Mathematics, Quest University, Squamish, BC V8B 0N8, Canada}
\email{nicole.magill@questu.ca}
\thanks{}
\author[Sarria]{Alejandro Sarria} 
\address{Department of Mathematics and Statistics, Williams College, Williamstown, MA 01267, USA}
\email{Alejandro.Sarria@williams.edu}
\thanks{This work was done at the 2016 Williams College SMALL REU program. The authors were partially supported by NSF REU grant DMS-1347804 and the Williams College Science Center. The first author was also partially supported by the Williams College Finnerty Fund.}


\begin{document}

\begin{abstract}
The generalized Hunter-Saxton system comprises several well-known models from fluid dynamics and serves as a tool for the study of fluid convection and stretching in one-dimensional evolution equations. In this work, we examine the global regularity of periodic smooth solutions of this system in $L^p$, $p \in [1,\infty)$, spaces for nonzero real parameters $(\lambda,\kappa)$. Our results significantly improve/extend those by Wunsch et al.~\cite{Wun09, Wun10, WW12} and Sarria~\cite{Sar15}. Furthermore, we study the effects that different boundary conditions have on the global regularity of solutions by replacing periodicity with a homogeneous three-point boundary condition and establish finite-time blowup of a local-in-time solution of the resulting system for particular values of the parameters.
\end{abstract}

\maketitle

\section{Introduction} \label{sec:intro}

We are concerned with the $L^p$, $1\leq p<+\infty$, regularity of solutions to the system 
\begin{equation} \label{eq:ghs}
\begin{cases}
u_{xt} + uu_{xx} - \lambda u_x^2 - \kappa \rho^2 = I(t), & x\in[0,1],\quad t > 0, \\
\rho_t + u\rho_x = 2\lambda \rho u_x, & x\in[0,1],\quad t > 0, \\
u(x,0) = u_0(x), \quad \rho(x,0) = \rho_0(x), & x \in [0,1],
\end{cases}
\end{equation}
where $\lambda$ and $\kappa$ are nonzero real parameters, the nonlocal term $I(t)$ is given by
\begin{equation} \label{eq:nonlocal}
I(t) = -\kappa \int_0^1 \rho^2 \dd x - (1 + \lambda) \int_0^1 u_x^2 \dd x,
\end{equation}
and solutions are subject to periodic boundary conditions
\begin{equation} \label{eq:pbc}
u(0,t) = u(1,t), \qquad u_x(0,t) = u_x(1,t), \qquad \rho(0,t) = \rho(1,t).
\end{equation}

System~\eqref{eq:ghs} was first introduced by Wunsch~\cite{Wun10} as the \emph{generalized Hunter-Saxton system} due to its connection, via $(\lambda,\kappa)=(-1/2,\pm1/2)$, to the Hunter-Saxton (HS) system. Both models have been studied extensively in the literature (see, e.g.,~\cite{WW12, Lenells2, Liu2, Moon13}, and references therein). The HS system is a particular case of the Gurevich-Zybin system describing the formation of large scale structure in the universe (c.f.~\cite{Pavlov1}). It also arises as the ``short-wave'' limit of the Camassa-Holm (CH) system~\cite{Aconstantin0, Escher0}
\begin{equation} \label{gCH}
\begin{cases}
m_t+2u_xm+um_x+k\rho\rho_x=0, &\quad k=\pm1, \\
\rho_t+(u\rho)_x=0, \\
m\equiv u-u_{xx},
\end{cases}
\end{equation}
which is in turn derived from the Green-Naghdi equations~\cite{Johnson1}, widely used in coastal oceanography to approximate the free-surface Euler equations. It is worth noting that for $\rho\equiv0$, the CH system~\eqref{gCH} reduces to the well-known CH equation, a nonlinear dispersive wave equation that arises in the study of propagation of unidirectional irrotational waves over a flat bed, as well as water waves moving over an underlying shear flow. The CH equation is completely integrable, has an infinite number of conserved quantities, and its solitary wave solutions are solitons~\cite{Camassa1}; it also admits ``peakons'' and ``breaking wave'' solutions. We direct the reader to~\cite{Camassa1, Dullin1, Aconstantin2, Aconstantin0, Escher0, Lenells2}, and references therein, for additional background and results on the CH equation and system.

Lastly, when $\rho\equiv0$ or $\rho=\sqrt{-1}\,u_x$, system~\eqref{eq:ghs} becomes the generalized inviscid Proudman-Johnson (giPJ) equation~\cite{Proudman1, Okamoto1, Wunsch1, SS13, SS15}, comprising:
\begin{itemize}
\item for $\lambda=-1$, the Burgers' equation of gas dynamics;
\item for $\lambda=\frac{1}{n-1}$, stagnation point-form solutions~\cite{Childress, Saxton1, Sarria3} of the $n-$dimensional incompressible Euler equations;
\item for $\lambda=-\frac{1}{2}$, the HS equation, describing the orientation of waves in massive director field nematic liquid crystals~\cite{Hunter1}. 
\end{itemize}

From a more heuristic point of view,~\eqref{eq:ghs} may serve as a tool to better understand one-dimensional fluid convection and stretching. More specifically, differentiating~\eqref{eq:ghs}i) in space and setting $\omega=-u_{xx}$ yields
\begin{equation} \label{eq:diff}
\begin{cases}
\;\;\omega_{t}+\underbrace{u\omega_{x}}_{\textup{convection}}+\;(1-2\lambda)\underbrace{\omega u_{x}}_{\textup{stretching}}+2\kappa\underbrace{\rho\rho_x}_{\textup{coupling}}=0,
\\
\;\;\rho_t+\underbrace{u\rho_x}_{\textup{convection}}=2\lambda u_x\rho.
\end{cases}
\end{equation}
The nonlinear terms in~\eqref{eq:diff}i) represent the competition in fluid convection between nonlinear steepening and amplification due to $(1-2\lambda)$-dimensional stretching and $2\kappa$-dimensional coupling~\cite{Holm1, Wun10}, with the parameters $\lambda$ and $\kappa$ measuring the ratio of stretching to convection and the impact of the coupling between $u$ and $\rho$, respectively.

\subsection{Previous results.} \label{subsec:prev}
Local well-posedness of~\eqref{eq:ghs} in particular Sobolev spaces has been established in~\cite{WW12}; see also~\cite{Wun10}. As for global well-posedness, Sarria~\cite{Sar15} investigated the $L^\infty([0,1])$ regularity of solutions arising from a large class of smooth initial data by deriving general representation formulae for solutions of~\eqref{eq:ghs}--\eqref{eq:pbc} along characteristics for $\lambda\neq0$. For convenience of the reader, we recall the representation formulae and the main results of~\cite{Sar15} below; for additional regularity results, the reader may refer to~\cite{Wun09, Wun10, WW12, ML12, Moon13}.

\smallskip

For as long as a solution exists, define characteristics $\gamma$ via the initial value problem
\begin{equation*} \label{eq:charIVP}
\gamma_t(x,t) = u(\gamma(x,t),t), \qquad \gamma(x,0) = x.
\end{equation*}
Then for $\lambda\neq0$, 
\begin{equation} \label{eq:ux}
u_x(\gamma(x,t),t) = \frac{\bar{\mathcal{P}}_0(t)^{-2\lambda}}{\lambda \eta(t)} \left\{ \frac{\mathcal J(x,t)}{\mathcal{Q}(x,t)} - \frac{1}{\bar{\mathcal{P}}_0(t)} \int_0^1 \frac{\mathcal J(x,t)}{\mathcal{Q}(x,t)^{1+\frac{1}{2\lambda}}} \dd x \right\}
\end{equation}
and
\begin{equation} \label{eq:rho}
\rho(\gamma(x,t),t) = \frac{\rho_0(x)}{\mathcal{Q}(x,t)} \left( \int_0^1 \frac{\mathrm{d}x}{\mathcal{Q}(x,t)^{\frac{1}{2\lambda}}} \right)^{-2\lambda},
\end{equation}
where
\begin{equation} \label{eq:P0}
\bar{\mathcal{P}}_0(t) = \int_0^1 \frac{\mathrm{d}x}{\mathcal{Q}(x,t)^{\frac{1}{2\lambda}}},\qquad 
\mathcal J(x,t) = 1 - \lambda \eta(t) u_0'(x)
\end{equation}
and
\begin{equation} \label{eq:Q}
\mathcal{Q}(x,t) = c(x)\eta(t)^2 - 2\lambda u_0'(x) \eta(t) + 1,\qquad c(x) = \lambda \left(\lambda u_0'(x)^2 - \kappa \rho_0(x)^2 \right).
\end{equation}
The strictly increasing function $\eta:[0,+\infty)\to[0,+\infty)$ satisfies the IVP
\begin{equation} \label{eq:etaivp}
\eta'(t) = \bar{\mathcal{P}}_0(t)^{-2\lambda}, \qquad \eta(0) = 0,
\end{equation}
and the ``Jacobian'', $\gamma_x$, is given by
\begin{equation} \label{eq:gammax}
\gamma_x(x,t) =\bar{\mathcal{P}}_0(t)^{-1}\mathcal{Q}(x,t)^{-\frac{1}{2\lambda}}.
\end{equation}
Lastly, the time variable $t$ obeys
\begin{equation} \label{eq:t}
 t(\eta) = \int_0^\eta \left( \int_0^1 \frac{\mathrm{d}x}{(c(x)\mu^2 - 2\lambda u_0'(x)\mu + 1)^{\frac{1}{2\lambda}}}\right)^{2\lambda} \dd \mu.
\end{equation}

\bigskip

\noindent Next we summarize the main results of~\cite{Sar15}; additional details on the qualitative behavior of solutions may be found therein.

\begin{thm}[global-in-time solutions] \label{thm:global}  
Suppose $u_0(x)$ and $\rho_0(x)$ are smooth and satisfy~\eqref{eq:pbc}. Then the solution $(u,\rho)$ of~\eqref{eq:ghs}--\eqref{eq:pbc} stays smooth for all time if any of the following hold:
\begin{enumerate}
\item $\lambda\kappa<0$ and $\rho_0$ never vanishes;
\item $\lambda\kappa<0$ and $\rho_0$ vanishes at finitely many points $x_i\in[0,1]$, $1\leq i\leq n$, with $\lambda u_0'(x_i)\leq0$;
\item $(\lambda,\kappa)\in(0,1]\times\mathbb{R}^-$ and $\rho_0$ vanishes at finitely many points $x_i\in[0,1]$, $1\leq i\leq n$, with $\lambda u_0'(x_i)>0$.
\end{enumerate}
\end{thm}

\begin{thm}[finite-time blowup for $\lambda\kappa<0$] \label{thm:blowupak<0}
Suppose $\rho_0(x_i)=0$ at finitely many points $x_i\in[0,1]$, $1\leq i\leq n$, and $\lambda u_0'(x_i)>0$. Without loss of generality, let $\max_{\{x_i\}}u_0'(x)=u_0'(x_1)$. Then there exist smooth initial data $(u_0(x), \rho_0(x))$ satisfying~\eqref{eq:pbc} such that
\begin{enumerate}
\item for $(\lambda,\kappa)\in(-2,0)\times\mathbb{R}^+$, $u_x$ undergoes ``one-sided, discrete'' blowup. In particular, there exists a finite $t_*>0$ such that $u_x(\gamma(x_1,t),t)\to-\infty$ as $t\nearrow t_*$, but remains finite otherwise;
\item for $(\lambda,\kappa)\in(-\infty,-2]\times\mathbb{R}^+$, $u_x$ undergoes ``two-sided, everywhere'' blowup. In particular, there exists a finite $t_*>0$ such that $u_x(\gamma(x_1,t),t)\to-\infty$ as $t\nearrow t_*$, while $u_x(\gamma(x,t),t)\to+\infty$ otherwise;
\item for $(\lambda,\kappa)\in(1,+\infty)\times\mathbb{R}^-$, $u_x$ undergoes two-sided, everywhere blowup.
\end{enumerate}
\end{thm}

\begin{thm}[finite-time blowup for $\lambda\kappa>0$] \label{thm:blowupak>0}
There exist smooth initial data $(u_0(x), \rho_0(x))$ satisfying~\eqref{eq:pbc} and a finite time $t_*>0$ such that
\begin{enumerate}
\item for $(\lambda,\kappa)\in(-1,0)\times\mathbb{R}^-$, $u_x$ undergoes one-sided discrete blowup as $t\nearrow t_*$, whereas, for $(\lambda,\kappa)\in(-\infty,-1]\times\mathbb{R}^-$, $u_x$ develops a two-sided, everywhere singularity;
\item for $(\lambda,\kappa)\in\mathbb{R}^-\times\mathbb{R}^-$, $\rho$ undergoes one-sided, discrete blowup as $t\nearrow t_*$;
\item for $(\lambda,\kappa)\in\mathbb{R}^+\times\mathbb{R}^+$,  $u_x$ undergoes two-sided, everywhere blowup, while $\rho$ develops a one-sided, discrete singularity.
\end{enumerate}
\end{thm}

\smallskip
\subsection{Summary of results} \label{subsec:summary}
Theorems~\ref{thm:global}--\ref{thm:blowupak>0} describe the $L^{\infty}([0,1])$ regularity of $u_x$ and $\rho$ for a large class of smooth, periodic initial conditions. The main purpose of this paper is to extend these results to $L^p([0,1])$ spaces for $p \in [1,\infty)$. We do this via a direct approach which involves using the representation formulae introduced in Section~\ref{subsec:prev} to compute the $L^p$ norm of the solution and then applying some standard $L^p$ space inequalities to the resulting expressions. 

In summary, we prove the following: 

\begin{thm}[$L^p$ regularity for $\lambda\kappa < 0$] \label{thm:ak<0}
Let $t_*>0$ denote the finite, $L^\infty$ blowup time in Theorem \ref{thm:blowupak<0} for $\lambda\kappa < 0$. There  exists smooth initial data $(u_0,\rho_0)$ satisfying~\eqref{eq:pbc} such that 
\begin{enumerate}
\item for $p \in (1,\infty)$ and $\lambda \in (-\infty,-2/p] \cup (1,\infty)$, $\lim_{t\nearrow t_*} \norm{u_x}_p = +\infty$.
\item for $p \in [1,\infty)$ and $\lambda \in (-2/(2p-1),0)$, $\lim_{t\nearrow t_*} \norm{u_x}_p < \infty$.
\end{enumerate}
\end{thm}

\begin{thm}[$L^p$ regularity for $\lambda\kappa > 0$] \label{thm:ak>0}
Let $t_*>0$ denote the finite, $L^\infty$ blowup time in Theorem \ref{thm:blowupak>0} for $\lambda\kappa > 0$. There exists smooth initial data $(u_0,\rho_0)$ satisfying~\eqref{eq:pbc} such that 
\begin{enumerate}
\item for $p \in (1,\infty)$ and $\lambda \in (-\infty,-1/p] \cup (0,\infty)$, $\lim_{t\nearrow t_*} \norm{u_x}_p = +\infty$.
\item for $p \in [1,\infty)$ and $\lambda \in (-1/(2p-1),0)$, $\lim_{t\nearrow t_*} \norm{u_x}_p < \infty$.
\item for $p \in [1,\infty)$ and $\lambda \in (-\infty,-1/(2p-1)] \cup (0,\infty)$, $\lim_{t\nearrow t_*} \norm{\rho}_p = +\infty$.
\item for $p \in [1,\infty)$ and $\lambda \in (-1/(2p-1),0)$, $\lim_{t\nearrow t_*} \norm{\rho}_p <\infty$.
\end{enumerate}
\end{thm}

Our results establish a connection between the $L^p([0,1])$ regularity of solutions and qualitative properties of $L^{\infty}([0,1])$ blowup: two-sided everywhere blowup of $u_x$ corresponds to its $L^1([0,1])$ norm escaping to infinity at the $L^{\infty}([0,1])$ blowup time, whereas, one-sided discrete blowup of $u_x$ corresponds to its containment, up to the $L^{\infty}([0,1])$ blowup time, in some $L^p([0,1])$ spaces for $p$ finite.

We also remark that our results are not exclusive to periodic solutions; they also hold, with the obvious modifications, under Dirichlet boundary conditions
\[ u(0,t)=u(1,t)=\rho(0,t)=\rho(1,t)=0. \]

To further investigate the role that the boundary conditions may play in the long-time behavior of solutions to~\eqref{eq:ghs}, we end the paper by considering the associated system
\begin{equation} \label{eq:ghs2}
\begin{cases}
u_{xxt}+uu_{xxx}+(1-2\lambda)u_xu_{xx}-2\kappa\rho\rho_x=0, & x\in[0,1],\quad t > 0, \\
\rho_t + u\rho_x = 2\lambda \rho u_x, & x\in[0,1],\quad t > 0, \\
u(x,0) = u_0(x), \quad \rho(x,0) = \rho_0(x), & x \in [0,1],
\end{cases}
\end{equation}
obtained by differentiating~\eqref{eq:ghs}i) with respect to $x$. Replacing~\eqref{eq:pbc} with the homogeneous, three-point boundary condition~\eqref{eq:ghs_3p}, we establish the following result for a local-in-time solution of~\eqref{eq:ghs2}.

\begin{thm} \label{thm:threept}
Suppose there exists $T>0$ such that $(u,\rho)$ solves~\eqref{eq:ghs2} with the homogeneous three-point boundary condition~\eqref{eq:ghs_3p} for all $0<t\leq T$. If $(\lambda,\kappa) \in (-1,+\infty) \times[0,+\infty)$, or $(\lambda,\kappa) \in (-\infty,-1) \times(-\infty,0]$, there exist smooth initial data $(u_0,\rho_0)$ such that $T<+\infty$. In particular, $|u(0,t)|\to+\infty$ as $t\nearrow T$.
\end{thm}

\smallskip
\subsection{Outline.} \label{subsec:outline}
The paper is organized as follows. In Section~\ref{sec:prelim}, we introduce the class of smooth initial data, derive bounds for the $L^p$ norm of $u_x$, and isolate two lemmas useful in approximating the asymptotic behavior of some of the integrals appearing in the norm expressions. These integral estimates are then used in Section~\ref{sec:Lpregularity} to prove Theorems~\ref{thm:ak<0} and \ref{thm:ak>0}. Lastly, Theorem~\ref{thm:threept} is established in Section~\ref{sec:threept}. For the convenience of the reader, we include an outline of the derivation~\cite{Sar15} of the representation formulae and some of the integral estimates used to prove the main theorems in Appendices~\ref{app:soln} and~\ref{app:int}, respectively.


\section{Preliminaries} \label{sec:prelim}

\subsection{The initial data} \label{subsec:blowup}

From the solution formulae~\eqref{eq:ux} and~\eqref{eq:rho}, we see that solutions potentially become infinite at the smallest positive $\eta$-value, say $\eta_*$, for which $\mathcal Q(x,t)$ vanishes for some $x \in [0,1]$. To this end, define sets
\begin{equation} \label{eq:omega}
\Omega = \{x \in [0,1] \mid c(x) = 0 \}
\end{equation}
and
\begin{equation} \label{eq:sigma}
\Sigma = \{x \in [0,1] \mid c(x) \neq 0 \}.
\end{equation}
Then $\mathcal Q$ admits three possible representations for each $x \in [0,1]$:
\begin{enumerate}
\item if $x \in \Omega$, then $\mathcal Q$ is linear in $\eta$:
\begin{equation} \label{eq:Q1}
\mathcal Q(x,t) = 1 - 2\lambda u_0'(x)\eta(t),
\end{equation}
\item if $x \in \Sigma$ and $\rho_0(x) = 0$, then $\mathcal Q$ is quadratic in $\eta$ and has a double root:
\begin{equation} \label{eq:Q2}
\mathcal Q(x,t) = (1 - \lambda\eta(t) u_0'(x))^2 = \mathcal J(x,t)^2,
\end{equation}
\item if $x \in \Sigma$ and $\rho_0(x) \neq 0$, then $\mathcal Q$ is quadratic in $\eta$ and has two single roots:
\begin{equation} \label{eq:Q3}
\mathcal Q(x,t) = (1 - \lambda \eta(t) g_+(x))(1 - \lambda \eta(t) g_-(x)),
\end{equation}
where
\begin{equation} \label{eq:gpm}
g_{\pm} = u_0'(x) \pm \sqrt\frac\kappa\lambda \abs{\rho_0(x)}.
\end{equation}
\end{enumerate}
Note that if $x \in \Sigma$, then the discriminant of $\mathcal Q$ with respect to $\eta$ is
\begin{equation*} \label{eq:disc}
\mathcal D(x) = 4\lambda\kappa \rho_0(x)^2.
\end{equation*}
To determine the locations where $\mathcal Q(x,t)$ vanishes earliest, we consider the cases $\lambda\kappa < 0$ and $\lambda\kappa > 0$ separately. \\

\noindent {\bf Case} $\lambda\kappa < 0$. \\
Suppose $\rho_0$ vanishes at finitely many points $\{x_1,\dots,x_n\} \subset [0,1]$, and that $\lambda u_0'(x_i) \neq 0$ for all $i = 1,\dots,n$ and $\lambda u_0'(x_i) > 0$ for some $i$. Note that $\Sigma = [0,1]$. Indeed, $c(x_i) =(\lambda u_0'(x_i))^2 \neq 0$ for all $i = 1,\dots,n$, and if $\rho_0(x) \neq 0$, then $c(x) = 0$ would imply
\begin{equation*}
0 \leq u_0'(x)^2 = \frac{\kappa}{\lambda} \rho_0(x)^2 < 0.
\end{equation*}
Hence $\mathcal Q$ is quadratic for each fixed $x \in [0,1]$. If $\rho_0(x) \neq 0$, then $\mathcal D(x) < 0$ and $\mathcal Q(x,0) = 1$ imply $0 < \mathcal{Q}(x,t) < \infty$ for all $\eta > 0$. From
\begin{equation*} \label{eq:Qak<0}
\mathcal Q(x_i,t) = (1 - \lambda \eta(t)u_0'(x_i))^2,
\end{equation*}
we see that if we set
\begin{equation} \label{eq:minmaxmM}
m_0 = \min_{x \in \{x_i\}} u_0'(x) < 0, \qquad\quad M_0 = \max_{x \in \{x_i\}} u_0'(x) > 0,
\end{equation}
then $\mathcal Q(x,t)$ vanishes first as
\begin{equation*} \label{eq:etastarmM}
\eta \nearrow \eta_* =
\begin{dcases}
\frac{1}{\lambda m_0} & \mbox{ if } \lambda < 0, \\
\frac{1}{\lambda M_0} & \mbox{ if } \lambda > 0.
\end{dcases}
\end{equation*}
Lastly, we assume that at the locations where $m_0$ (resp.\ $M_0$) is achieved, $u_0''(x) = 0$ and $u_0'''(x) > 0$ (resp.\ $u_0'''(x) < 0$). Then the $L^{\infty}([0,1])$ blowup time 
\begin{equation*}
t_* \equiv \lim_{\eta \uparrow \eta_*} t(\eta),
\end{equation*}
for $t(\eta)$ as in~\eqref{eq:t}, is finite for $\lambda \in (-\infty,0) \cup (1, \infty)$; see Theorem \ref{thm:blowupak<0}. \\

\noindent{\bf Case} $\lambda\kappa > 0$. \\
If $\lambda\kappa > 0$, then $\Omega$ is not necessarily empty, so all three representations~\eqref{eq:Q1}--\eqref{eq:Q3} of $\mathcal Q$ are possible. We remark that the case where $\eta_*$ is a double root of $\mathcal Q$ has been studied extensively in connection with the generalized inviscid Proudman-Johnson equation~\cite{SS13, SS15}. We direct the reader to these works for results in this direction, as well as for the special case where $Q$ is identically linear, i.e., $c(x)\equiv0$. Thus, for $\lambda\kappa>0$, we only treat the simple case where $\eta_*$ is a single root arising from~\eqref{eq:Q3}.\footnote[1]{If $\eta_*$ occurs as a root of~\eqref{eq:Q1}, a nearly identical arguments may be used.} 

For $\lambda < 0$, and respectively $\lambda > 0$, set
\begin{equation} \label{eq:minmaxnN}
n = \min_{\substack{x \in \Omega \\ \rho_0(x) \neq 0}} g_-(x) < 0, \qquad\quad N =  \max_{\substack{x \in \Omega \\ \rho_0(x) \neq 0}} g_+(x) > 0,
\end{equation}
and assume that at the finitely many locations where $n$ (resp.\ $N$) is achieved, the smooth initial data is such that $g_-$ (resp.\ $g_+$) has vanishing first order derivative and non-vanishing second order derivative. Then $\mathcal Q(x,t)$ vanishes earliest as
\begin{equation*} \label{eq:etastarnN}
\eta \nearrow \eta_* =
\begin{dcases}
\frac{1}{\lambda n} & \mbox{ if } \lambda < 0, \\
\frac{1}{\lambda N} & \mbox{ if } \lambda > 0,
\end{dcases}
\end{equation*}
and solutions to~\eqref{eq:ghs}--\eqref{eq:pbc} leave $L^{\infty}([0,1])$ in finite time (see Theorem \ref{thm:blowupak>0}).

\subsection{$L^p$ estimates} \label{subsec:Lpnorm}

In this section, we begin our study of the $L^p$, $p\in[1,+\infty)$, regularity of solutions to~\eqref{eq:ghs}--\eqref{eq:pbc} under the setup described in Section~\ref{subsec:blowup}. Using the representation formulae along characteristics, we explicitly compute the $p$-norms of $u_x$ and $\rho$. We then provide rudimentary upper and lower bounds for the $p$-norm of $u_x$, which, as we will see later, are sufficient to classify the $L^p$ regularity of solutions for much of the $\lambda$-$\kappa$ parameter space. We conclude with two general lemmas that will allow us to estimate, near the $L^{\infty}$ blowup time, the behavior of particular integral terms appearing in the bounds.

For as long as solutions exist, the Jacobian $\gamma_x$~\eqref{eq:gammax} is an increasing diffeomorphism of the unit circle $\mathbb S$ to itself~\cite{Wun10}. Hence
\begin{equation} \label{eq:uxnorm}
\begin{split}
\norm{u_x(\, \cdot \,, t)}_p
&= \left( \int_0^1 \left|u_x(\gamma(x,t),t)\right|^p \gamma_x \dd x \right)^{\frac1p} \\
&= \frac{\bar{\mathcal{P}}_0^{-2\lambda - \frac1p}}{\abs{\lambda}\eta}  \left( \int_0^1  \left| \frac{\mathcal{J}}{\mathcal{Q}^{1+\frac{1}{2\lambda p}}} - \frac{1}{\bar{\mathcal{P}}_0} \cdot \frac{1}{\mathcal{Q}^{\frac{1}{2\lambda p}}} \int_0^1 \frac{\mathcal{J}}{\mathcal{Q}^{1+\frac{1}{2\lambda}}} \dd \alpha  \right|^p \dd x \right)^{\frac1p}
\end{split}
\end{equation}
and
\begin{equation} \label{eq:rhonorm}
\begin{split}
\norm{\rho(\, \cdot \,, t)}_p^p
&= \int_0^1 \left|\rho(\gamma(x,t),t)\right|^p \gamma_x \dd x \\
&= \bar{\mathcal P}_0^{-2\lambda p-1} \int_0^1 \abs{\rho_0(x)}^p \frac{1}{\mathcal Q^{p + \frac{1}{2\lambda}}} \dd x. 
\end{split}
\end{equation}
By the convexity of the $p$-norm and Jensen's inequality,
\begin{equation} \label{eq:uxlower}
\norm{u_x(\,\cdot\,,t)}_p
\geq \frac{\bar{\mathcal{P}}_0^{-2\lambda - \frac1p}}{\abs{\lambda}\eta}  \left| \int_0^1   \frac{\mathcal{J}}{\mathcal{Q}^{1 + \frac{1}{2\lambda p}}} \dd x - \frac{1}{\bar{\mathcal{P}}_0}  \int_0^1 \frac{\mathrm{d} x}{\mathcal{Q}^{\frac{1}{2\lambda p}}}  \int_0^1 \frac{\mathcal{J}}{\mathcal{Q}^{1+\frac{1}{2\lambda}}} \dd x \right|,
\end{equation}
and by Minkowski's inequality,
\begin{equation} \label{eq:uxupper}
\norm{u_x(\,\cdot\,,t)}_p
\leq \frac{\bar{\mathcal{P}}_0^{-2\lambda - \frac1p}}{\abs{\lambda}\eta} \left[ \left( \int_0^1   \frac{\mathcal{J}^p}{\mathcal{Q}^{p+\frac{1}{2\lambda}}} \dd x \right)^{\frac{1}{p}} + \frac{1}{\bar{\mathcal{P}}_0^{1-\frac1p}} \int_0^1 \frac{\mathcal{J}}{\mathcal{Q}^{1+\frac{1}{2\lambda}}} \dd x  \right].
\end{equation}
Although the right-hand side of~\eqref{eq:uxlower} vanishes for $p = 1$, it does allow us to investigate the $L^p$ regularity of $u_x$ for $p \in (1,\infty)$. To prove Theorems \ref{thm:ak<0} and \ref{thm:ak>0}, we must first estimate the blowup rates of the integrals appearing in~\eqref{eq:rhonorm}--\eqref{eq:uxupper}. To this end, we use the following two lemmas, whose proofs are deferred to Appendix~\ref{app:int}.

\begin{lemma} \label{lem:estlamneg}
Suppose $h(x)\in C^2([0,1])$ attains a local minimum $m < 0$ at finitely many points $\underline{x} \in (0,1)$ with $h''(\underline{x}) > 0$. Set $\tau_* = \frac{1}{\lambda m}$ for $\lambda < 0$ and let $0<\tau<\tau_*$. Then for $\delta > 0$ and $\tau_* - \tau > 0$ both small,
\begin{equation} \label{eq:lamneg}
 \int_{\underline{x}-\delta}^{\underline{x}+\delta} \frac{\mathrm dx}{(1 - \lambda \tau h(x))^b} \sim 
  \begin{cases}
  C_2(1 - \lambda \tau m)^{\frac12 - b} & \mbox{if } b > 1/2, \\
  -C_3\log(1 - \lambda \tau m) & \mbox{if } b = 1/2, \\
  C & \mbox{if } b < 1/2,
  \end{cases}
\end{equation}
where $C_2$ and $C_3$ are positive constants given by
\begin{equation*} \label{eq:lamnegconst}
C_1 = \frac{h''(\underline{x})}{2}, \qquad C_2 = \frac{\Gamma\left(b - \frac12\right)}{\Gamma\left(b\right)} \sqrt\frac{\abs{m}\pi}{C_1}, \qquad C_3 = \frac{\abs{m}}{\sqrt{C_1}}, 
\end{equation*}
$C$ is a positive constant that depends only on $\lambda$ and $b$, and $\Gamma(\cdot)$ is the standard gamma function.
\end{lemma}

\begin{lemma} \label{lem:estlampos}
Suppose $h(x) \in C^2([0,1])$ attains a local maximum $M > 0$ at finitely many points $\bar{x} \in (0,1)$ with $h''(\bar{x}) < 0$. Set $\tau_* = \frac{1}{\lambda M}$ for $\lambda > 0$ and let $0<\tau<\tau_*$. Then for $\delta > 0$ and $\tau_* - \tau > 0$ both small,
\begin{equation} \label{eq:lampos}
\int_{\bar{x}-\delta}^{\bar{x}+\delta} \frac{\mathrm{d} x}{(1 - \lambda \tau h(x))^b} \sim 
  \begin{cases}
  C_5(1 - \lambda \tau M)^{\frac12 - b} & \mbox{if } b > 1/2, \\
  -C_6\log(1 - \lambda \tau M) & \mbox{if } b = 1/2, \\
  C & \mbox{if } b < 1/2,
\end{cases}
\end{equation}
where $C_5$ and $C_6$ are positive constants given by
\begin{equation*} \label{eq:lamposconst}
C_5 = \frac{h''(\bar{x})}{2}, \qquad C_5 = \frac{\Gamma\left(b - \frac12\right)}{\Gamma\left(b\right)} \sqrt\frac{M\pi}{\abs{C_4}}, \qquad C_6 = \frac{M}{\sqrt{C_4}}, 
\end{equation*}
and $C$ is a positive constant that depends only on $\lambda$ and $b$.
\end{lemma}

For the convenience of the reader, we illustrate two applications of the lemmas and estimate the blowup rate of
\begin{equation*}
\int_0^1 \frac{\mathcal J}{\mathcal Q^{1 + \frac{1}{2\lambda}}} \dd x
\end{equation*}
for $\lambda\kappa < 0$ and $\lambda > 1$, and then for $\lambda\kappa>0$ with $\lambda > 0$.

Recall that for $\lambda\kappa < 0$ and $\lambda > 1$, the earliest root $\eta_*$ of the quadratic~\eqref{eq:Q2} occurs at locations where $M_0$ (as defined in~\eqref{eq:minmaxmM}ii)) is achieved; for simplicity, assume $M_0$ occurs at a single point $\bar{x}$. For $\eta_* = \frac{1}{\lambda M_0}$ and $\eta_* - \eta > 0$ small,~\eqref{eq:lampos}i) yields
\begin{equation*}
\int_0^1 \frac{\mathcal J}{\mathcal Q^{1+\frac{1}{2\lambda}}} \dd x
\sim
\int_{\bar{x} - \delta}^{\bar{x} + \delta} \frac{1}{(1 - \lambda\eta(t)u_0'(x))^{1+\frac{1}{\lambda}}} \dd x
\sim
C_2(1 - \lambda \eta M_0)^{-\frac12 - \frac1\lambda}.
\end{equation*}
For $\lambda\kappa > 0$ and $\lambda > 0$, the earliest root $\eta_*$ occurs as a single root of~\eqref{eq:Q3} at locations where $N$ (as defined in ~\eqref{eq:minmaxnN}ii)) is achieved. Again assume for simplicity that $N$ is achieved at a single point $\bar{x}$. Then
\begin{equation*}
\int_0^1 \frac{\mathcal J}{\mathcal Q^{1+\frac{1}{2\lambda}}} \dd x
\sim
\int_{\bar{x} - \delta}^{\bar{x} + \delta} \frac{C}{(1-\lambda\eta g_+(x))^{1+\frac{1}{2\lambda}}} \dd x
\sim
C_5(1 - \lambda \eta N)^{-\frac12 - \frac{1}{2\lambda}}
\end{equation*}
for $\eta_* = \frac{1}{\lambda N}$ and $\eta_* - \eta>0$ small.

\section{Proofs of the theorems} \label{sec:Lpregularity}

We are now ready to prove the Theorems. Throughout the proofs, $C$ will denote a generic positive constant that may change in value from line to line.

%
%
%
%
%

\begin{proof}[Proof of Theorem \textup{\ref{thm:ak<0}}] 
Let $M_0 > 0 > m_0$ be as in~\eqref{eq:minmaxmM}, and set
\begin{equation*} \label{eq:etastarmM2}
\eta_* =
\begin{dcases} 
\frac{1}{\lambda m_0} & \mbox{ if } \lambda < 0, \\
\frac{1}{\lambda M_0} & \mbox{ if } \lambda > 0.
\end{dcases}
\end{equation*}

\medskip

\noindent For $\lambda < 0$ and $1 \leq p < \infty$, Lemma~\ref{lem:estlamneg} yields

\begin{align*}  
\int_0^1 \frac{\mathcal J}{\mathcal Q^{1 + \frac{1}{2\lambda p}}} \dd x
&\sim 
\begin{cases}
C_2 (1 - \lambda \eta m_0)^{-\frac12 - \frac{1}{\lambda p}},\quad & \lambda \in (-\infty,-2/p), \\
-C_3 \log(1 - \lambda \eta m_0),  & \lambda = -2/p, \\
C,  & \lambda \in (-2/p, 0),
\end{cases} 
\\
\int_0^1 \frac{\mathcal J}{\mathcal Q^{1 + \frac{1}{2\lambda}}} \dd x
&\sim 
\begin{cases}
C_2 (1 - \lambda \eta m_0)^{-\frac12 - \frac{1}{\lambda}},\quad & \lambda \in (-\infty,-2), \\
-C_3 \log(1 - \lambda \eta m_0),  & \lambda = -2, \\
C,  & \lambda \in (-2, 0),
\end{cases} 
\\
\int_0^1 \frac{1}{\mathcal Q^{\frac{1}{2\lambda p}}} \dd x
&\sim C, \qquad \lambda \in (-\infty,0), 
\\
\bar{\mathcal P}_0 
&\sim C, \qquad \lambda \in (-\infty,0), 
\\
\int_0^1 \frac{\mathcal J^p}{\mathcal Q^{p+\frac{1}{2\lambda}}} \dd x
&\sim 
\begin{cases}
C_2 (1 - \lambda \eta m_0)^{\frac12 - \frac1\lambda - p},\quad & \lambda \in (-\infty, -2/(2p-1)), \\
-C_3 \log(1 - \lambda \eta m_0), & \lambda = -2/(2p-1), \\
C, & \lambda \in (-2/(2p-1),0)
\end{cases} 
\end{align*}

\medskip
\noindent for $\eta - \eta_* > 0$ small. Similarly for $\lambda > 1$ and $1 \leq p < 2$,

\begin{align*}  
\int_0^1 \frac{\mathcal J}{\mathcal Q^{1 + \frac{1}{2\lambda p}}} \dd x
&\sim
C_5 (1 - \lambda \eta M_0)^{-\frac12 - \frac{1}{\lambda p}}, \qquad \lambda \in (1,\infty), 
\\
\int_0^1 \frac{\mathcal J}{\mathcal Q^{1 + \frac{1}{2\lambda}}} \dd x 
&\sim
C_5 (1 - \lambda \eta M_0)^{-\frac12 - \frac{1}{\lambda}}, \qquad \lambda \in (1,\infty), 
\\
\int_0^1 \frac{1}{\mathcal Q^{\frac{1}{2\lambda p}}} \dd x
&\sim
\begin{cases}
C_5 (1 - \lambda \eta M_0)^{\frac12 - \frac{1}{\lambda p}},\quad & \lambda \in (1,2/p), \\
-C_6 \log(1 - \lambda \eta M_0),  & \lambda = 2/p, \\
C,  & \lambda \in (2/p, \infty).
\end{cases} 
\\
\bar{\mathcal P}_0 
&\sim
\begin{cases}
C_5 (1 - \lambda \eta M_0)^{\frac12 - \frac{1}{\lambda}},\quad & \lambda \in (1,2), \\
-C_6 \log(1 - \lambda \eta M_0), & \lambda = 2, \\
C,  & \lambda \in (1, \infty).
\end{cases} \label{eq:P0ak<0a>0}  
\end{align*}

\bigskip

\noindent Applying these estimates to~\eqref{eq:uxlower} and~\eqref{eq:uxupper}, we obtain, as $\eta \nearrow \eta_*$, the following: \\
For $\lambda < -2$,  
\begin{equation*}
\norm{u_x(\,\cdot\,,t)}_p \geq
C \abs{\frac{1}{(1 - \lambda \eta m_0)^{\frac12 + \frac{1}{\lambda p}}} - \frac{C}{(1 - \lambda \eta m_0)^{\frac12 + \frac{1}{\lambda}}}} \to +\infty.
\end{equation*}
For $\lambda = -2$,  
\begin{equation*}
\norm{u_x(\,\cdot\,,t)}_p \geq
C \abs{\frac{1}{(1 - \lambda \eta m_0)^{\frac12 + \frac{1}{\lambda p}}} + C \log (1 - \lambda \eta m_0)} \to +\infty.
\end{equation*}
For $-2 < \lambda < -2/p$, 
\begin{equation*}
\norm{u_x(\,\cdot\,,t)}_p \geq
C \abs{\frac{1}{(1 - \lambda \eta m_0)^{\frac12 + \frac{1}{\lambda p}}} - C} \to +\infty.
\end{equation*}
For $\lambda = -2/p$, 
\begin{equation*}
\norm{u_x(\,\cdot\,,t)}_p \geq C \abs{\log(1-\lambda\eta m_0) + C} \to +\infty.
\end{equation*}
For $-2/(2p-1) < \lambda < 0$,  
\begin{equation*}
\norm{u_x(\,\cdot\,,t)}_p \to C < +\infty.
\end{equation*}
For $1 < \lambda < 2$, choose $p' \leq p$ such that $1 < p' < 5/3$ and $\frac{3 p'-1}{2 p'}<\lambda <\frac{2}{p'}$. Then $\lambda + \frac{1}{2p'} - \frac{3}{2}>0$, so that
\begin{equation*}  
\norm{u_x(\,\cdot\,,t)}_{p} \geq \norm{u_x(\,\cdot\,,t)}_{p'} \geq \frac{C}{(1-\lambda\eta M_0)^{\lambda + \frac{1}{2 p'}-\frac{3}{2}}} \to +\infty.
\end{equation*}
For $\lambda = 2$,  
\begin{equation*}
\begin{split}
\norm{u_x(\,\cdot\,,t)}_p
&\geq \frac{C}{\abs{\log(1-2\eta M_0)}^{4+\frac1p}} \abs{\frac{1}{(1-2\eta M_0)^{\frac12 + \frac{1}{2 p}} }+ \frac{C}{(1-2\eta M_0)\log(1-2\eta M_0)}} \\
&= \frac{C}{(1-2\eta M_0)\abs{\log(1-2\eta M_0)}^{5+\frac1p}} \abs{(1-2\eta M_0)^{\frac{1}{2}-\frac{1}{2p}} \log(1-2\eta M_0) + C} \to +\infty.
\end{split}
\end{equation*}
For $\lambda > 2$,  
\begin{equation*}
\norm{u_x(\,\cdot\,,t)}_p \geq
C \abs{\frac{1}{(1 - \lambda\eta M_0)^{\frac12 + \frac{1}{\lambda p}}} - \frac{C}{(1 - \lambda\eta M_0)^{\frac12 + \frac1\lambda}}} \to +\infty.
\end{equation*}
\end{proof}

\begin{proof}[Proof of Theorem \textup{\ref{thm:ak>0}}]
Let $N > 0 > n$ be as in~\eqref{eq:minmaxnN}, and set
\begin{equation*} \label{eq:etastarnN2}
\eta_* =
\begin{dcases}
\frac{1}{\lambda n} & \mbox{ if } \lambda < 0 \\
\frac{1}{\lambda N} & \mbox{ if } \lambda > 0. 
\end{dcases}
\end{equation*}
\medskip

\noindent For $\lambda < 0$ and $1 \leq p < \infty$, Lemma~\ref{lem:estlamneg} implies that

\begin{align*}  
\int_0^1 \frac{\mathcal J}{\mathcal Q^{1 + \frac{1}{2\lambda p}}} \dd x
&\sim 
\begin{cases}
C_2 (1 - \lambda \eta n)^{-\frac12 - \frac{1}{2\lambda p}},\quad & \lambda \in (-\infty,-1/p), \\
-C_3 \log(1 - \lambda \eta n),  & \lambda = -1/p, \\
C,  & \lambda \in (-1/p, 0),
\end{cases}  
\\
\int_0^1 \frac{\mathcal J}{\mathcal Q^{1 + \frac{1}{2\lambda}}} \dd x
&\sim 
\begin{cases}
C_2 (1 - \lambda \eta n)^{-\frac12 - \frac{1}{2\lambda}},\quad & \lambda \in (-\infty,-1), \\
-C_3 \log(1 - \lambda \eta n),  & \lambda = -1, \\
C,  & \lambda \in (-1, 0).
\end{cases}  
\\
\int_0^1 \frac{1}{\mathcal Q^{\frac{1}{2\lambda p}}} \dd x
&\sim C, \qquad \lambda \in (-\infty,0),  
\\
\bar{\mathcal P}_0 =
&\sim C, \qquad \lambda \in (-\infty,0),  
\\
\int_0^1 \frac{\mathcal J^p}{\mathcal Q^{p+\frac{1}{2\lambda}}} \dd x
&\sim 
\begin{cases}
C_2(1 - \lambda \eta n)^{\frac12 - p - \frac1{2\lambda}},\quad & \lambda \in (-\infty, -1/(2p-1)), \\
-C_3 \log(1 - \lambda \eta n), & \lambda = -1/(2p-1), \\
C, & \lambda \in (-1/(2p-1),0),
\end{cases}  
\\
\int_0^1 \frac{1}{\mathcal Q^{p+\frac{1}{2\lambda}}} \dd x
&\sim
\begin{cases}
C_2(1 - \lambda \eta n)^{\frac12 - \frac1{2\lambda} - p},\quad & \lambda \in (-\infty, -1/(2p-1)), \\
-C_3 \log(1 - \lambda \eta n), & \lambda = -1/(2p-1), \\
C, & \lambda \in (-1/(2p-1),0)
\end{cases}  
\end{align*}

\medskip
\noindent for $\eta_* - \eta > 0$ small. Similarly, for $\lambda > 0$ and $1 \leq p < \infty$,

\begin{align*}  
\int_0^1 \frac{\mathcal J}{\mathcal Q^{1 + \frac{1}{2\lambda p}}} \dd x
&\sim
C_5 (1 - \lambda \eta N)^{-\frac12 - \frac{1}{2\lambda p}}, \qquad \lambda \in (0,\infty), 
\\
\int_0^1 \frac{\mathcal J}{\mathcal Q^{1 + \frac{1}{2\lambda}}} \dd x 
&\sim
C_5 (1 - \lambda \eta N)^{-\frac12 - \frac{1}{2\lambda}}, \qquad \lambda \in (0,\infty), 
\\
\int_0^1 \frac{1}{\mathcal Q^{\frac{1}{2\lambda p}}} \dd x
&\sim
\begin{cases}
C_5 (1 - \lambda \eta N)^{\frac12 - \frac{1}{2\lambda p}},\quad & \lambda \in (0,1/p), \\
-C_6 \log(1 - \lambda \eta N), & \lambda = 1/p, \\
C,  & \lambda \in (1/p, \infty),
\end{cases} 
\\
\bar{\mathcal P}_0 = 
&\sim
\begin{cases}
C_5 (1 - \lambda \eta N)^{\frac12 - \frac{1}{2\lambda}},\quad & \lambda \in (0,1), \\
-C_6 \log(1 - \lambda \eta N), & \lambda = 1, \\
C,  & \lambda \in (1, \infty),
\end{cases} 
\\
\int_0^1 \frac{1}{\mathcal Q^{p+\frac{1}{2\lambda}}} \dd x
&\sim
C_5(1 - \lambda\eta N)^{\frac12 - \frac{1}{2\lambda} - p}, \qquad \lambda \in (0,\infty). 
\end{align*}

\bigskip

\noindent Applying these estimates to~\eqref{eq:uxlower} and~\eqref{eq:uxupper}, we obtain, as $\eta \nearrow \eta_*$, the following: \\
If $\lambda < -1$,  
\begin{equation*}
\norm{u_x(\,\cdot\,,t)}_p \geq C\abs{\frac{1}{(1-\lambda\eta n)^{\frac12+\frac{1}{2\lambda p}}} -\frac{C}{(1-\lambda\eta n)^{\frac12+\frac{1}{2\lambda}}}} \to +\infty.
\end{equation*}
If $\lambda = -1$,  
\begin{equation*}
\norm{u_x(\,\cdot\,,t)}_p \geq C\abs{\frac{1}{(1-\lambda\eta n)^{\frac12+\frac{1}{2\lambda p}}} + C\log(1-\lambda\eta n)} \to +\infty.
\end{equation*}
If $-1 < \lambda < -1/p$,  
\begin{equation*}
\norm{u_x(\,\cdot\,,t)}_p \geq C \abs{\frac{1}{(1-\lambda\eta n)^{\frac12+\frac{1}{2\lambda p}}} - C} \to +\infty.
\end{equation*}
If $\lambda = -1/p$,  
\begin{equation*}
\norm{u_x(\,\cdot\,,t)}_p \geq C \abs{\log(1-\lambda\eta n) + C} \to +\infty.
\end{equation*}
If $-1/(2p-1) < \lambda < 0$,  
\begin{equation*}
\norm{u_x(\,\cdot\,,t)}_p \to C < \infty.
\end{equation*}
If $0 < \lambda < 1$, choose $p' \leq p$ such that $1 < p' < 3$ and $\frac{p'-1}{2p'} < \lambda < \frac1{p'}$. Then $\lambda + \frac{1}{2p'} - \frac12 > 0$, so that
\begin{equation*}   
\norm{u_x(\,\cdot\,,t)}_{p} \geq \norm{u_x(\,\cdot\,,t)}_{p'} \geq \frac{C}{(1-\lambda\eta N)^{\lambda+\frac{1}{2p'}-\frac12}} \to +\infty.
\end{equation*}
If $\lambda = 1$,  
\begin{equation*}
\begin{split}
\norm{u_x(\,\cdot\,,t)}_p
&\geq \frac{C}{(-\log(1-\lambda\eta N))^{2 + \frac1p}} \abs{ \frac{1}{(1-\lambda\eta N)^{\frac12 + \frac1{2 p}}}+ \frac{C}{(1-\lambda\eta N)\log (1-\lambda\eta N)}} \\
&= \frac{C}{(1-\lambda\eta N)\abs{\log(1-\lambda\eta N)}^{3+\frac1p}} \abs{ (1-\lambda\eta N)^{\frac12 - \frac{1}{2p}} \log(1-\lambda \eta N) + C} \to +\infty.
\end{split}
\end{equation*}
If $\lambda > 1$,  
\begin{equation*}
\norm{u_x(\,\cdot\,,t)}_p \geq C \abs{ \frac{1}{(1-\lambda\eta N)^{\frac12+\frac{1}{2\lambda p}}} - \frac{C}{(1-\lambda\eta N)^{\frac12 + \frac{1}{2\lambda}}}} \to +\infty.
\end{equation*}
Finally, for $\eta_* - \eta > 0$ small,~\eqref{eq:rhonorm} yields
\begin{equation*}
\norm{\rho(\,\cdot\,,t)}_p^p
\sim
\begin{cases}
C(1 - \lambda \eta n)^{\frac12 - \frac1{2\lambda} - p},\quad & \lambda \in (-\infty, -1/(2p-1)), \\
C \abs{\log(1 - \lambda \eta n)}, & \lambda = -1/(2p-1), \\
C, & \lambda \in (-1/(2p-1),0).
\end{cases}
\end{equation*}
for $\lambda < 0$ and
\begin{equation*}
\norm{\rho(\,\cdot\,,t)}_p^p
\sim
\begin{cases}
C(1-\lambda\eta N)^{-\lambda p}, & \lambda \in (0,1), \\
C(1-\lambda\eta N)^{-p}  \left|\log(1 - \lambda \eta n)\right|^{-2p - 1},\quad & \lambda = 1, \\
C(1-\lambda\eta N)^{\frac12 - \frac1{2\lambda} - p}, & \lambda \in (1,\infty),
\end{cases}
\end{equation*}
for $\lambda > 0$. This completes the proof.
\end{proof}


\section{Three-point boundary conditions}\label{sec:threept}

In this section, we establish finite time blowup of a local-in-time solution of the IVP
\begin{equation} \label{eq:ghs3}
\begin{cases}
u_{xxt}+uu_{xxx}+(1-2\lambda)u_xu_{xx}-2\kappa\rho\rho_x=0, & x\in[0,1],\quad t > 0, \\
\rho_t + u\rho_x = 2\lambda \rho u_x, & x\in[0,1],\quad t > 0, \\
u(x,0) = u_0(x), \quad \rho(x,0) = \rho_0(x), & x \in [0,1],
\end{cases}
\end{equation}
with the homogeneous three-point boundary condition
\begin{equation} \label{eq:ghs_3p}
u(1,t) = u_x(0,t) = u_x(1,t) = 0, \qquad \rho(1,t) = 0
\end{equation}
for particular values of $\lambda$ and $\kappa$ (see Theorem~\ref{thm:threept}). For smooth $(u_0,\rho_0)$ satisfying the appropriate boundary condition, our result implies the existence of parameter values $(\lambda,\kappa)$ such that solutions to~\eqref{eq:ghs}--\eqref{eq:pbc} stay smooth for all time, whereas those of~\eqref{eq:ghs3}--\eqref{eq:ghs_3p} blowup in finite time.

\begin{proof}[Proof of Theorem \textup{\ref{thm:threept}}]

Multiplying~\eqref{eq:ghs3}i) by $x$ and integrating by parts yields
\begin{equation*}
0 = -\ddt \int_0^1 u_x \dd x + (1 + \lambda) \int_0^1 u_x^2 \dd x + \kappa \int_0^1 \rho^2 \dd x
\end{equation*}
Set
\begin{equation*}
H(t) := \int_0^1 u_x \dd x = -u(0,t), 
\end{equation*}
so that
\begin{equation*}
H'(t) = (1 + \lambda) \norm{u_x}_2^2 + \kappa \norm{\rho}_2^2.
\end{equation*}
Note that by the Cauchy-Schwarz inequality,
\begin{equation*}
H(t)^2 \leq \norm{u_x}_2^2.
\end{equation*}
Then for $\lambda > -1$ and $\kappa \geq 0$,
\begin{equation*} \label{eq:3pricatti}
H'(t) \geq (1 + \lambda)H(t)^2 \geq 0.
\end{equation*}
Thus choosing $H(0) > 0$ we see that, as long as the solution exists, $H(t) > 0$. Integrating then yields
\begin{equation*}
0 < \frac{1}{H(t)} \leq \frac{1}{H(0)} - (1 + \lambda)t,
\end{equation*}
whose right-hand side vanishes as $t$ approaches the finite time $T=\frac{1}{(1+\lambda)H(0)}$. The case $(\lambda,\kappa) \in (-\infty,-1) \times (-\infty,0]$ can be handled similarly.
\end{proof}

\begin{appendix}


\section{Solving on characteristics} \label{app:soln}

In this section we summarize the derivation of the solution formulae for~\eqref{eq:ghs}--\eqref{eq:pbc} along characteristics for arbitrary $(\lambda,\kappa) \in \R\setminus\{0\} \times \R$. We use the notation in Section~\ref{sec:intro} for any auxiliary functions. For detailed computations, the reader is referred to Section~2 of~\cite{Sar15}. In brief, the method of characteristics is used to write $u_x \circ \gamma$ and $\rho \circ \gamma$ in terms of the Jacobian $\gamma_x$. The system is reformulated as a second-order nonlinear ODE for $\gamma_x$, which is then solved via reduction of order.

Fix $x \in [0,1]$, and define, for as long as $u$ exists, characteristics $\gamma$ via the initial value problem
\begin{equation} \label{eq:characteristics}
	\dot\gamma(x,t)=u(\gamma(x,t),t), \qquad \gamma(x,0)=x,
\end{equation}
where $\dot{} = \ddt$. Differentiating in space yields
\begin{equation} \label{eq:gamma_xt}
	\dot\gamma_{x}=u_x(\gamma(x,t),t)\cdot \gamma_x, \qquad \gamma_x(x,0)=1,
\end{equation}
so that
\begin{equation} \label{eq:gammax_deriv}
	\gamma_x = e^{\int_0^tu_x(\gamma(x,s),s) \dd s}
\end{equation}
Differentiating $\rho \circ \gamma$ in time and using~\eqref{eq:ghs}ii) immediately gives
\begin{equation} \label{eq:rhochar}
	\rho(\gamma(x,t),t) = \rho_0(x) \cdot e^{2\lambda\int_0^tu_x(\gamma(x,s),s) ds} = \rho_0(x) \cdot \gamma_x^{2\lambda}.
\end{equation}
Now, it follows from~\eqref{eq:gamma_xt} and~\eqref{eq:rhochar} that~\eqref{eq:ghs}i) along characteristics is given by
\begin{equation} \label{eq:uxGamma}
\ddt\left(u_x(\gamma(x,t),t)\right)=\lambda\left(\dot\gamma_{x}\gamma_{x}^{-1}\right)^2+\kappa\left(\rho_0\cdot\gamma_x^{2\lambda}\right)^2+I(t).
\end{equation}
Differentiating~\eqref{eq:gamma_xt} in time, substituting into~\eqref{eq:uxGamma}, and re-arranging yields the following second-order nonlinear ODE for $\omega(x,t) = \gamma_x(x,t)^{-\lambda}$:
\begin{equation} \label{eq:ODEnonhomo}
	\ddot w(x,t)+\lambda I(t)\omega(x,t)=-\lambda\kappa\rho_0(x)^2\omega^{-3}.
\end{equation}
First consider the corresponding linear homogeneous ODE
\begin{equation} \label{eq:ODEhomo}
	\ddot y(x,t)+\lambda I(t)y(x,t)=0.
\end{equation}
Suppose $\phi_1(t)$ and $\phi_2(t)$ are two linearly independent solutions to~\eqref{eq:ODEhomo} satisfying $\phi_1(0)=\dot\phi_2(0)=1$ and $\dot\phi_1(0)=\phi_2(0)=0$. Then the general solution to~\eqref{eq:ODEhomo} is of the form
\begin{equation} \label{eq:gensol}
y(x,t) = c_1(x)\phi_1(t) + c_2(x)\phi_2(t),
\end{equation}
where, by reduction of order,
\begin{equation} \label{eq:reduction}
	\phi_2(t)=\phi_1(t)\eta(t), \qquad \eta(t)=\int_0^t\frac{\mathrm{d}s}{\phi_1(s)^2}.
\end{equation}
Since $\dot\omega = -\lambda \gamma_x^{-\lambda-1}\dot\gamma_x$ and $\gamma_x(x,0) = 1$, it follows that $\omega(x,0) = 1$ and $\dot\omega(x,0) = -\lambda u_0'(x)$, from which the coefficients $c_1(x)$ and $c_2(x)$ may be obtained. This reduces~\eqref{eq:gensol} to
\begin{equation} \label{eq:yreduce}
y(x,t) = \phi_1(t) \mathcal J(x,t),
\end{equation}
Turning to the nonhomogeneous equation~\eqref{eq:ODEnonhomo}, write
\begin{equation} \label{eq:omegaxt}
\omega(x,t) = z(\eta(t))y(x,t),
\end{equation}
for some function $z$ with $z(0) = 1$ and $z'(0) = 0$. Plugging~\eqref{eq:omegaxt} into~\eqref{eq:ODEnonhomo} and yields the following IVP for $\mu(x,\eta) = z(\eta)\mathcal J(x,t)$:
\begin{equation} \label{eq:mu}
\begin{cases}
\mu_{\eta\eta} = -\lambda\kappa\rho_0(x)^2 \mu^{-3}, \\
\mu(x,0) = 1, \quad \mu_\eta(x,0) = -\lambda u_0'(x).
\end{cases}
\end{equation}
Writing~\eqref{eq:mu} as a first-order equation, solving for $\mu_\eta$, integrating, and solving for $z(\eta)$ yields
\begin{equation}
z(\eta) = \frac{\mathcal Q(x,t)}{\mathcal J(x,t)^2}.
\end{equation}
It follows from $\omega(x,t) = \gamma_x(x,t)^{-\lambda}$ and equations~\eqref{eq:yreduce} and~\eqref{eq:omegaxt} that
\begin{equation} \label{eq:gammaxphi1}
\gamma_x(x,t) = \left[ \phi_1(t)^2 \mathcal Q(x,t) \right]^{-\frac{1}{2\lambda}}.
\end{equation}
It remains to determine $\phi_1$. Note that by the uniqueness of solutions to~\eqref{eq:characteristics} and periodicity that
\begin{equation} \label{eq:gammaperiodic}
\gamma(x+1,t) = 1 + \gamma(x,t)
\end{equation}
for all $x \in [0,1]$. Integrating~\eqref{eq:gammaxphi1} in space therefore yields
\begin{equation} \label{eq:phi1}
\phi_1(t) = \bar{\mathcal{P}}_0(t)^\lambda.
\end{equation}
It follows that
\begin{equation} \label{eq:gammaxfinal}
\gamma_x = \bar{\mathcal{P}}_0^{-1} \mathcal Q^{-\frac{1}{2\lambda}}.
\end{equation}
The resulting expressions~\eqref{eq:ux} and~\eqref{eq:rho} for $u_x$ and $\rho$ follow from~\eqref{eq:gammaxfinal},~\eqref{eq:gammax_deriv}, and~\eqref{eq:rhochar}.


\section{Integral estimates} \label{app:int}

In this section we briefly outline the technique used to estimate the spatial integrals in~\eqref{eq:uxlower} and~\eqref{eq:uxupper} as $\eta \nearrow \eta_*$. The proofs of Lemmas \ref{lem:estlamneg} and \ref{lem:estlampos} are identical so we only prove the former. The estimates are based on a Taylor approximation, with the $b \geq 1/2$ case following by straightforward integration. For the $0 < b < 1/2$ case, we make use of the Gauss hypergeometric series (see~\cite{hypergeo1,hypergeo2,hypergeo3})
\begin{equation} \label{eq:hypergeo}
{}_2F_1\left[a,b;c;z\right]\equiv\sum_{k=0}^{\infty}\frac{\left(a\right)_k(b)_k}{\left(c\right)_k\,k!}z^k,\qquad \lvert z\rvert< 1
\end{equation} 
defined for $c\notin\mathbb{Z}^-\cup\{0\}$ and $(x)_k,\, k\in\mathbb{N}\cup\{0\}$, the Pochhammer symbol
\begin{equation}
(x)_0=1, \qquad (x)_k=x(x+1)\cdots(x+k-1).
\end{equation}
We also need the following two results:
\begin{lemma}[see~\cite{hypergeo2,hypergeo3}] \label{lem:analcont}
Suppose $\lvert\text{arg}\left(-z\right)\rvert<\pi$ and $a,b,c,a-b\notin\mathbb{Z}.$ The analytic continuation for $\lvert z\rvert>1$ of the series (\ref{eq:hypergeo}) is given by 
\begin{equation} \label{eq:analform}
\begin{split}
{}_2F_1[a,b;c;z]=&\frac{\Gamma(c)\Gamma(a-b)(-z)^{-b}{}_2F_1[b,1+b-c;1+b-a;z^{-1}]}{\Gamma(a)\Gamma(c-b)}
\\
&+\frac{\Gamma(c)\Gamma(b-a)(-z)^{-a}{}_2F_1[a,1+a-c;1+a-b;z^{-1}]}{\Gamma(b)\Gamma(c-a)}.
\end{split}
\end{equation} 
\end{lemma}
\begin{lemma}[see~\cite{SS13}] \label{lem:hyper}
Suppose $b\in(-\infty,2)\backslash\{1/2\},\,\,0\leq\left|\beta-\beta_0\right|\leq1$ and $\epsilon\geq C_0$ for some $C_0>0.$ Then 
\begin{equation} \label{eq:derseries}
\frac{1}{\epsilon^b}\,\frac{d}{d\beta}\left((\beta-\beta_0){}_2F_1\left[\frac{1}{2},b;\frac{3}{2};-\frac{C_0(\beta-\beta_0)^2}{\epsilon}\right]\right)=(\epsilon+C_0(\beta-\beta_0)^2)^{-b}.
\end{equation} 
\end{lemma}

\begin{proof}[Proof of Lemma \textup{\ref{lem:estlamneg}}]
A Taylor expansion about $\underline{x}$ yields
\begin{equation} \label{eq:htaylor}
\eps + h(x) - m \sim \eps + C_1(x - \underline{x})^2.
\end{equation}
For $\eps > 0$ small and $b > 1/2$,
\begin{equation} \label{eq:intapprox1}
\begin{split}
\int_{\underline{x} - \delta}^{\underline{x} + \delta} \frac{\mathrm{d}x}{\left(\eps + h(x) - m\right)^b}
&\sim \int_{\underline{x} - \delta}^{\underline{x} + \delta} \frac{\mathrm{d}x}{(\eps + C_1\left(x - \underline{x}\right)^2)^b} \\
&= \frac{1}{\eps^b} \int_{\underline{x} - \delta}^{\underline{x} + \delta} \frac{\mathrm{d}x}{(1 + ( \sqrt{C_1/\eps}(x - \underline{x}))^2)^b} \\
&\sim \frac{2}{\eps^{b-\frac12}\sqrt{C_1}} \int_0^{\pi/2} \cos^{2b-2}\theta \dd\theta \\
&=  \frac{\Gamma \left(b-\frac{1}{2}\right)}{\Gamma (b)} \sqrt{\frac{\pi}{C_1}} \cdot \frac{1}{\eps^{b-\frac12}}.
\end{split}
\end{equation}
Setting $\eps = m - \frac{1}{\lambda\tau}$, we find that for $\tau_* - \tau > 0$ small,
\begin{equation} \label{eq:intapprox2}
\int_{I} \frac{\mathrm{d} x}{(1 - \lambda \tau h(x))^b} \sim \frac{C_2}{(1 - \lambda\tau m)^{b - \frac12}}.
\end{equation}
The case $b = 1/2$ follows by a similar argument. Estimate~\eqref{eq:lamneg}iii) follows trivially if $b \leq 0$; to establish the estimate for $0 < b < \frac12$, we use Lemmas \ref{lem:analcont} and \ref{lem:hyper}. The Taylor approximation~\eqref{eq:htaylor} and Lemma \ref{lem:hyper} imply that
\begin{equation} \label{eq:intapprox3}
\begin{split}
\int_{\underline{x}-\delta}^{\underline{x}+\delta} \frac{\mathrm{d}x}{(\eps + h(x) - m)^b}
&\sim
\int_{\underline{x}-\delta}^{\underline{x}+\delta} \frac{\mathrm{d}x}{(\eps + C_1(x - \underline{x})^2)^b} \\
&=
\frac{2\delta}{\eps^b} {}_2F_1 \left[ \frac12,b,\frac32,-\frac{C_1\delta^2}{\eps} \right]
\end{split} 
\end{equation}
for $\eps \geq C_1\geq \delta^2C_1 > 0$, i.e. $-1 \leq -\frac{\delta^2C_1}{\epsilon} < 0$. If we let $\eps > 0$ become sufficiently small enough, so that $-\frac{\delta^2C_1}{\eps} < -1$, then the analytic continuation formula yields
\begin{equation} \label{eq:analcontapprox}
\frac{2\delta}{\eps^b} {}_2F_1\left[\frac{1}{2}, b; \frac32;-\frac{C_1\delta^2}{\eps}\right]
= \frac{2\delta^{1-2b}}{(1-2b)C_1^b} + \frac{\Gamma\left(b-\frac12\right)}{\Gamma\left(b\right)} \sqrt{\frac{\pi}{C_1}} \eps^{b-\frac12}  + \psi(\eps)
\end{equation}
for $\psi(\epsilon)=\textsl{o}(1)$ as $\eps \to 0$. Since $1-2b>0$,
\begin{equation} \label{eq:intapprox4}
\int_{\underline{x}-\delta}^{\underline{x}+\delta} \frac{\mathrm{d}x}{(\eps + C_1(x - \underline{x})^2)^b}
\sim
\frac{\Gamma\left(b-\frac12\right)}{\Gamma\left(b\right)} \sqrt{\frac{\pi}{C_1}} \eps^{b-\frac12}.
\end{equation}
Setting $\eps = m - \frac{1}{\lambda\tau}$ in~\eqref{eq:intapprox4}, we obtain 
\begin{equation} \label{eq:intapprox5}
\int_{\underline{x}-\delta}^{\underline{x}+\delta} \frac{\mathrm{d}x}{(1 - \lambda\tau h(x))^b} \sim C
\end{equation}
for $\tau_*-\tau>0$ small.
\end{proof}

\end{appendix}

\end{document}